\documentclass[12pt]{article}
\usepackage[utf8]{inputenc}
\usepackage{geometry}
\usepackage{amsmath, amsthm, amssymb} 
\usepackage{graphicx}
\usepackage[dvipsnames]{xcolor}
\graphicspath{ {./images/} }
\usepackage{float}
\usepackage[normalem]{ulem}
\usepackage{hyperref}
\hypersetup{
    colorlinks=true,
    linkcolor=magenta,
    filecolor=magenta,      
    urlcolor=magenta,
    citecolor=magenta
} 

\theoremstyle{plain}
\newtheorem{proposition}{Proposition}
\newtheorem{theorem}{Theorem}

\newtheorem{definition}{Definition}
\newtheorem{lemma}{Lemma}

\newtheorem{example}{Example}
\newtheorem{remark}{Remark}

\newcommand{\R}{\mathbb{R}}
\newcommand{\C}{\mathbb{C}}
\newcommand{\cset}{\mathcal{C}_\Omega}
\newcommand{\cdisc}{\mathcal{D}_\Omega}

\title{Catastrophe in elastic tensegrity frameworks}
\author{Alexander Heaton and Sascha Timme}
\date{\today}

\begin{document}

\maketitle

\begin{abstract}
    We discuss elastic tensegrity frameworks made from rigid bars and elastic cables, depending on many parameters.
    For any fixed parameter values, the stable equilibrium position of the framework is determined by minimizing an energy function subject to algebraic constraints.
    As parameters smoothly change, it can happen that a stable equilibrium disappears. This loss of equilibrium is called \emph{catastrophe} since the framework will experience large-scale shape changes despite small changes of parameters. Using nonlinear algebra we characterize a semialgebraic subset of the parameter space, the \emph{catastrophe set}, which detects the merging of local extrema from this parametrized family of constrained optimization problems, and hence detects possible catastrophe. Tools from numerical nonlinear algebra allow reliable and efficient computation of all stable equilibrium positions as well as the catastrophe set~itself.
\end{abstract}

\section{Introduction}\label{section:introduction}

Tensegrity structures appear in nature and engineering, scaling in size from nanometers \cite{LHTIS2010} to meters \cite{TP2003}, used on the earth \cite{M2003, SO2009} and in outer space \cite{T2002, ZGSMPTDG2012}. Since the tension in the lightweight cables provides stability \cite{C1978, ZO2015}, they can hold their shape without any locking mechanisms. This and other advantages make tensegrity highly appealing for \textit{deployable structures} \cite{P2001}. They can significantly change size and shape, using several different functional configurations during their application.
        
In this article, we discuss \textit{elastic tensegrity frameworks} (Definition \ref{definition:elastic-tensegrity-framework}) made from rigid bars and elastic cables, similar to those appearing in \cite{LieWuPaulinoQi2017ProgrammableDeploymentTensegrityStructures, SternJayaramMurugan2018ShapingTopologyFoldingPathways}, but also similar to the \textit{tensegrity frameworks} defined in \cite{ConnellyWhiteley1996} which are popular in the mathematics and combinatorics literature. Instead of edge length inequalities as in \cite{ConnellyWhiteley1996}, we use Hooke's law to introduce an energy function that distinguishes between bars and elastic cables. The configuration is then determined by solving a constrained optimization problem.
This provides a large family of simple models which are effectively treated using the theory of elasticity and energy minimization (see Definition \ref{definition:stable-elastic-tensegrity-framework}). We use numerical nonlinear algebra to calculate \textit{all} equilibrium positions, in contrast to the more widely-used iterative methods (e.g. Newton-Raphson) which can only find one solution at a time, with no guarantees on finding them~all.

Elastic tensegrity frameworks depend on many parameters, e.g., the length of its rigid bars or the fixed position of some nodes. For a given framework we can choose a space of control parameters $\Omega$ whose values are viewed as the parameters we can manipulate. A \textit{path} is a map from the unit interval $y:[0,1]\subset \mathbb{R} \to \Omega$ which describes how the controls $y(t)$ vary in time. We use numerical nonlinear algebra to track the changes in stable equilibrium positions of the framework as the control parameters vary. Most importantly, we are interested in a positive-dimensional semialgebraic subset $\cset \subset \Omega$ called the \textit{catastrophe set} (Definition \ref{definition:catastrophe-set-D-Omega}). This set records those values of control parameters whose crossing could result in a discontinuous jump in the location of the nearest local equilibrium, since the current equilibrium can disappear after crossing $\cset$. This loss of equilibrium and the resulting behavior is called a \textit{catastrophe}. The importance of this set is well-known (see \cite{Arnold1986CatastropheTheory} for an overview), but we find that studying it from the algebraic perspective provides useful benefits.

Therefore, the purpose of this article is to show how techniques from numerical nonlinear algebra can be used to compute the catastrophe set $\cset$. For this we introduce an algebraic reformulation in Section \ref{section:algebraic-from-the-start} that we use to compute a superset $\cdisc \supset \cset$ which contains the relevant information for the original problem (Section \ref{section:non-algebraic-setup-elastic-tensegrity-frameworks}). This algebraic set $\cdisc$, the \emph{catastrophe discriminant}, detects the merging of equilibrium solutions from a parametrized family of constrained optimization problems.

Hooke's law provides a simple model which has proven extremely effective in an enormous amount of real-world situations. Also, in the article \textit{The Catastrophe Controversy} \cite{Guckenheimer-catastrophe-controversy} Guckenheimer writes ``The application of Catastrophe -
Singularity Theory to problems of elastic stability has been the greatest
success of the theory thus far.'' Thus catastrophe discriminants are of known importance, but they are very difficult to explicitly compute and this has limited their usefulness. With the development of efficient techniques in numerical nonlinear algebra, explicit computation of catastrophe discriminants is now within reach. Therefore, another purpose of this article is to explicitly describe these computations for a family of simple models (elastic tensegrity frameworks) which will be useful in many different applications.

\begin{figure}[!htb]
    \centering
    \includegraphics[width=\textwidth]{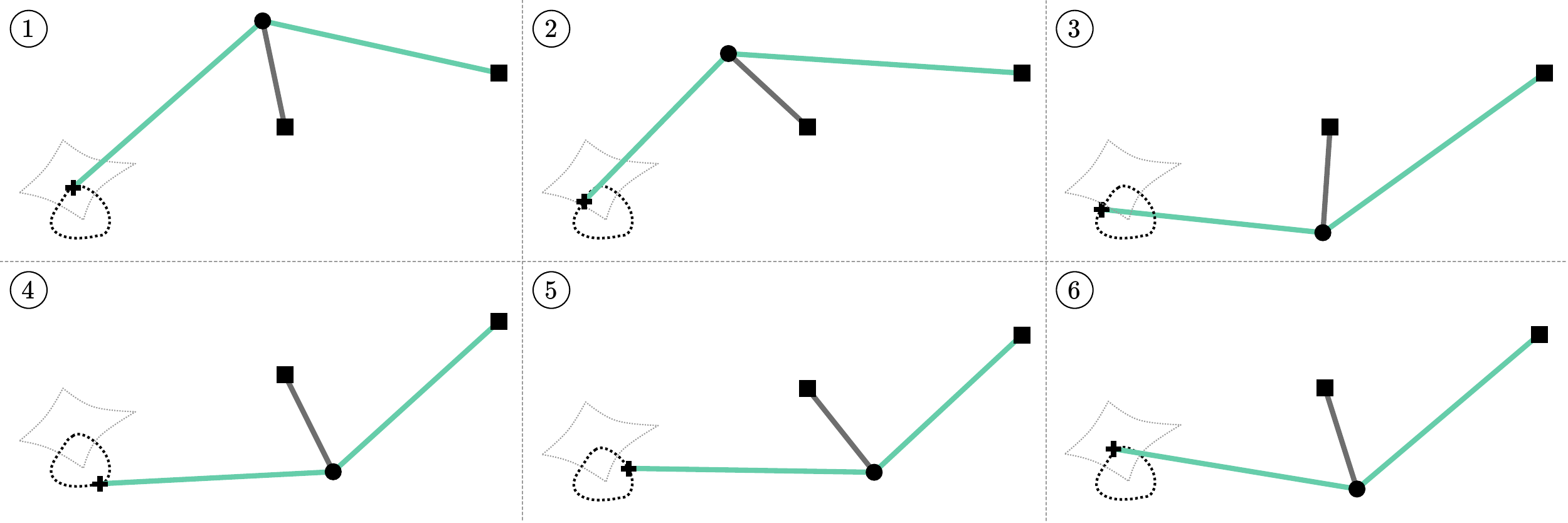}
    \caption{Loop crossing the catastrophe set. The black edge is a rigid bar and the green edges are elastic cables. Square nodes have fixed positions, the cross node is controlled around a loop, and the circular node's position is determined by minimizing the potential energy in the green elastic cables.}
    \label{figure:zeemans-machine-loop}
\end{figure}

A running example, simple enough to understand yet complicated enough to illustrate the advantage of knowing $\cset$, is \textit{Zeeman's catastrophe machine}.
Zeeman's catastrophe machine consists of a rigid bar which can rotate freely around one of its endpoints. Attached to the non-fixed endpoint are two elastic cables.
The end of one of the cables is fixed, the other can be moved freely.
The machine and its behavior is depicted in Figure \ref{figure:zeemans-machine-loop} at six discrete-time snapshots.
For more on this example see \cite{PostonWoodcock1973ZeemansCatastropheMachine}, where they give a parametrization of $\cset$ for a simplified machine, and implicit equations defining $\cset$ for the actual machine. See also \cite[Section 4]{Arnold1986CatastropheTheory}. In contrast, we use sample points to encode $\cset$ not only for Zeeman's machine but for any elastic tensegrity framework. The basic idea of Zeeman's machine is to control the free endpoint $y(t) \in \Omega \simeq \mathbb{R}^2$ of one cable while the rotating rigid bar settles into a position of minimum energy.
Using numerical nonlinear algebra we can reliably compute all complex solutions to this constrained optimization problem and find among them the real-valued and stable local minima. In addition, we compute a pseudo-witness set \cite{HAUENSTEIN20103349} for $\cdisc$ allowing effective sampling of the catastrophe set $\cset$, and therefore easily computable information on when catastrophes may occur, and how to avoid them entirely.

For those readers new to Zeeman's machine, consider the behavior depicted in Figure~\ref{figure:zeemans-machine-loop}. The black bar can rotate around its base, as the green elastic cables pull on its free endpoint.
As one of the cable endpoints moves smoothly, the stable equilibrium position of the machine also moves smoothly... usually. Upon crossing $\cset$ it can happen that this stable equilibrium disappears. This forces the machine to rapidly change shape, moving towards some new equilibrium. 
Without knowledge of $\cset$, these behaviors can be very surprising. For example, moving the control point in a small loop does not ensure a return to the original position for the machine (see Figure \ref{figure:zeemans-machine-loop}). 
Playing with this example, one quickly discovers the advantages of knowing $\cset$. Seemingly random catastrophes become easily predictable. 

\begin{figure}[ht]
    \centering
    \includegraphics[width=\textwidth]{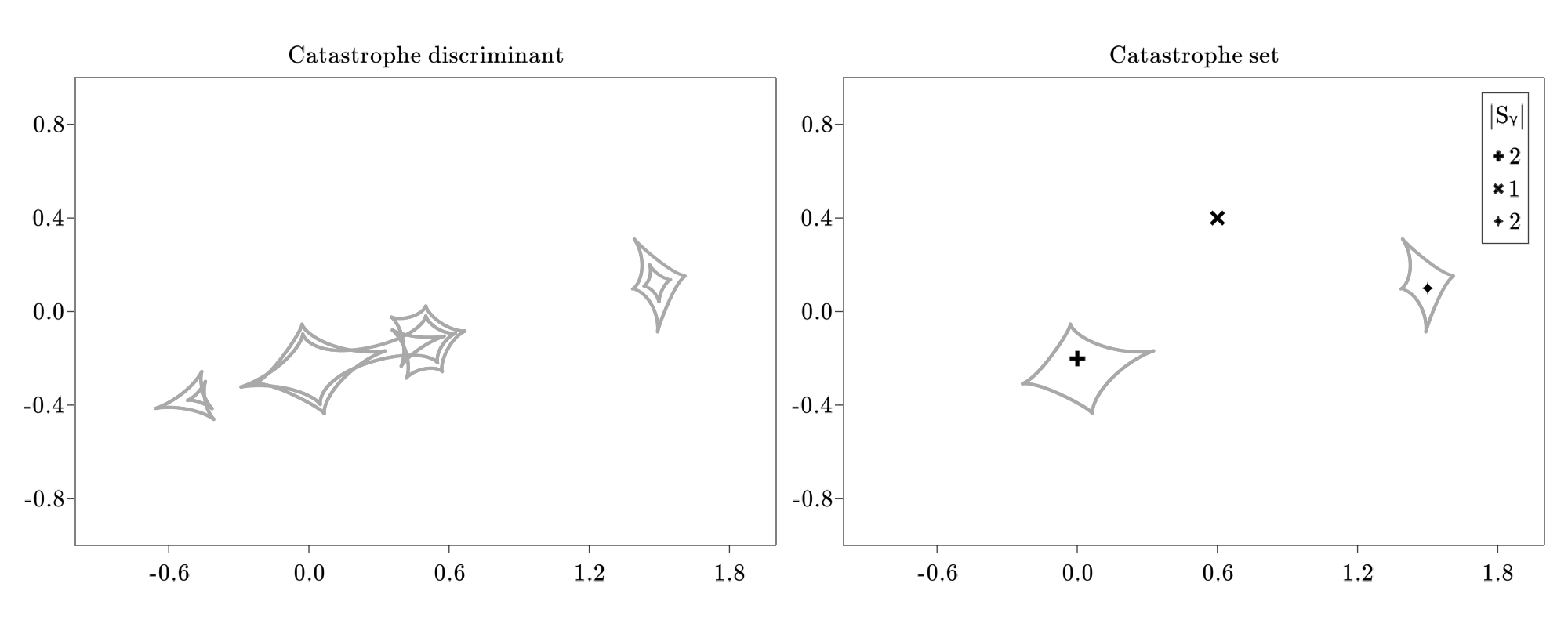}
    \caption{Catastrophe discriminant $\cdisc$ (left, degree 72) and catastrophe set $\cset$ (right) for Zeeman's machine, sampled numerically using homotopy continuation methods.}
    \label{figure:computed-catastrophe-curves}
\end{figure}

Section \ref{section:non-algebraic-setup-elastic-tensegrity-frameworks} gives the basic definitions for elastic tensegrity frameworks. In Section \ref{section:algebraic-from-the-start} we describe an algebraic reformulation of the relevant energy minimization problem. In so doing we naturally arrive at the \textit{equilibrium degree} of an elastic tensegrity framework (Definition \ref{definition:equilibrium-degree}), and the \textit{catastrophe degree} of its catastrophe discriminant (Definition \ref{definition:catastrophe-degree}). These numbers are intrinsic to the algebraic approach and characterize the algebraic complexity of each elastic tensegrity framework for a dense set of control parameters. Though the algebraic approach naturally deals with the algebraic set $\cdisc$, the original problem deals with the semialgebraic set $\cset$ (Definition \ref{definition:catastrophe-set-D-Omega}). For Zeeman's machine, both sets are shown in Figure~\ref{figure:computed-catastrophe-curves}. We note that $\cset$ in Figure~\ref{figure:computed-catastrophe-curves} is the \textit{envelope} of a family of curves, each of which is a \textit{conchoid of Nicomedes} \cite{Klein1895FamousProblemsElementaryGeometry, PostonWoodcock1973ZeemansCatastropheMachine}. 
Section 3 finishes by proving the main Theorem \ref{theorem:path-lifting-correct}, which shows that to any control path $y(t)$ that avoids $\cset$ there corresponds a unique path of stable configurations of the elastic tensegrity framework. 
In Section \ref{section:algebraic-computations} we give more details on the required computations using numerical nonlinear~algebra. Finally, in Section \ref{section:elastic-four-bar} we demonstrate our newly developed tools on a four-bar linkage, which easily becomes an elastic tensegrity framework upon the attachment of two elastic cables (Figure \ref{fig:fourbar-abstract}). We compute both $\cdisc$ and $\cset$ (Figure \ref{fig:fourbar_catastrophes}) and explicitly demonstrate one possible catastrophe (Figure \ref{fig:fourbar_before_after}).
Code that reproduces all examples in this article can be found at \url{https://doi.org/10.5281/zenodo.4056121}.

\section{Elastic tensegrity frameworks}\label{section:non-algebraic-setup-elastic-tensegrity-frameworks}

In this section we formally introduce elastic tensegrity frameworks and the necessary definitions and concepts to talk about their equilibrium positions.
Let $G = ([n],E)$ be a graph on $[n] := \{1,2,\dots,n\}$ nodes and $E = B \cup C$ edges.
Edges are two-element subsets of $[n]$.
Every $ij \in B$ is a rigid bar between nodes $i$ and $j$ and we have $\ell_{ij}$ as its length.
Similarly, every $ij \in C$ is an elastic cable between nodes $i$ and $j$
that has natural resting length $r_{ij}$ and a constant of elasticity $c_{ij}$.
The graph $G$ is embedded by a map $p:[n] \to \R^d$ and we denote the coordinates of the $n$ nodes of $G$ by $p_1=(p_{11},\ldots,p_{1d}),\dots,p_n \in \R^d$ and identify the space of coordinates with $\R^{nd}$.

\begin{example}[Zeeman's catastrophe machine]
\label{ex:zeeman-1}
We illustrate the definitions and concepts of this and the next section on Zeeman's catastrophe machine.
Zeeman's machine is an elastic tensegrity framework on $n=4$ nodes with edges $E = \{14,24,34\}$ partitioned as $B=\{14\}$ and $C=\{24,34\}$. See Figure \ref{fig:zeeman-catastrophe-setup} for an illustration.
\end{example}

\begin{figure}[hb]
    \centering
    \includegraphics[width=0.5\textwidth]{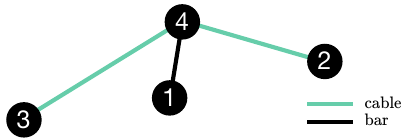}
    \caption{Our setup of Zeeman's catastrophe machine}
    \label{fig:zeeman-catastrophe-setup}
\end{figure}

For every rigid bar $ij \in B$ we define the bar constraint polynomial
\begin{equation}\label{equation:bar-constraints-bij}
    b_{ij} := \sum_{k \in [d]} (p_{ik} - p_{jk})^2 - \ell_{ij}^2
\end{equation}
and denote by $b$
the polynomial system whose component functions are the $b_{ij}$ for $ij \in B$.
For each elastic cable $ij \in C$ we define its potential energy $q_{ij}$ using Hooke's law
\begin{equation}\label{equation:energy-fcn-Q-with-square-roots}
    q_{ij} := \frac12 c_{ij} \left( \text{max} \left\{ 0, \,\, \sqrt{ \sum_{k \in [d]} (p_{ik} - p_{jk})^2 } - r_{ij} \right\} \right)^2 \hspace{1cm} \text{with} \hspace{1cm} Q =  \sum_{ij \in C} q_{ij}\,.
\end{equation}
This says that the energy $q_{ij}$ is proportional to the square of the distance the elastic cable has been stretched past its natural resting length. Though we have only introduced rigid bars and elastic cables, one could easily add compressed elastic edges which want to expand according to Hooke's law. For ease of exposition we proceed with elastic cables and rigid bars, rather than also including compressive struts in our notation.
 
We have introduced several variables. As shorthand we use the symbols $p,\ell,r,c$ to refer to the variables
\begin{equation}\label{definition:variable-names}
    \begin{array}{l}
        p_{ik} \hspace{1cm} \text{for } i \in [n], k \in [d] \\
        \ell_{ij} \hspace{1cm} \text{for } ij \in B \\
        r_{ij} \hspace{1cm} \text{for } ij \in C \\
        c_{ij} \hspace{1cm} \text{for } ij \in C.
    \end{array}
\end{equation}
In various examples some of these variables will be viewed as \textit{control parameters} $y \in Y \simeq \R^{m_1}$ whose values we can fix or manipulate at will, while the other variables will be viewed as \textit{internal parameters} $x \in X \simeq \R^{m_2}$ whose values are determined by the controls $y$ and the principle of energy minimization.
Often we may fix several control parameters and let others vary in some subset $\Omega \subset Y$.

\begin{example}[Zeeman's catastrophe machine (continued)]
\label{ex:zeeman-2}
We continue with Example \ref{ex:zeeman-1}.
We choose $X,Y$ as
\begin{align*}
    X = \left\{ \, (p_{41}, p_{42}) \, \right\} &\simeq \mathbb{R}^2\\
    Y = \left\{ \, (p_{11}, p_{12}, p_{21}, p_{22}, p_{31}, p_{32}, \ell_{14}, r_{24}, r_{34}, c_{24}, c_{34}) \, \right\} &\simeq \mathbb{R}^{11}
\end{align*}
but only consider the subset $\Omega \subset Y$ as in
\begin{equation*}
    \Omega = \left\{ \,\, (0,0,2,-1,p_{31},p_{32},1,1,1,0.5, 0.5) \,\, : \,\, (p_{31},p_{32}) \in \mathbb{R}^2 \right\} \subset Y.
\end{equation*}
In this setup, we have fixed everything except the coordinates of nodes $3$ and $4$. We \textit{control} the $y = (p_{31},p_{32}) \in \Omega$ and \textit{solve for} the $ x = (p_{41},p_{42}) \in X$. This means that for a given $y = (p_{31},p_{32}) \in \Omega$ we find the coordinates $ x = (p_{41},p_{42}) \in X$ which minimize
\begin{multline*}
    Q(p_{41},p_{42}) = \frac{1}{4} \text{max}\left\{ 0, \sqrt{{\left(2 - p_{41}\right)}^{2} + {\left(-1 - p_{42}\right)}^{2}} - 1 \right\} \\
    + \frac{1}{4} \text{max}\left\{ 0, \sqrt{{\left(p_{31} - p_{41}\right)}^{2} + {\left(p_{32} - p_{42}\right)}^{2}} - 1 \right\}
\end{multline*}
restricted to the set $\{ (x,y) : b(x,y) = 0 \} \cap (X \times \Omega )$. In this case, since $B = \{ 14 \}$ the constraints $b(x,y) = 0$ have only one equation $b_{14}(x,y)=0$ which reads
\begin{equation*}
    b_{14}(x,y) = {\left(0 - p_{41}\right)}^{2} + {\left(0 - p_{42}\right)}^{2} - 1^2 = 0.
\end{equation*}
\end{example}

\begin{definition}\label{definition:elastic-tensegrity-framework}
An \textit{elastic tensegrity framework} is a graph on nodes $[n]$ with edges $E \subset \binom{[n]}{2}$ along with the energy function $Q$ of (\ref{equation:energy-fcn-Q-with-square-roots}), a partition $E = B \cup C$ of the edge set into rigid bars and elastic cables, and a partition of variables $p,\ell,r,c$ of (\ref{definition:variable-names}) into internal and control parameters $X$ and $\Omega \subset Y$. A \textit{configuration} of an elastic tensegrity framework is a tuple $(x,y) \in X \times Y$ satisfying the bar constraints $b(x,y) = 0$ from \eqref{equation:bar-constraints-bij}.
\end{definition}

\begin{remark}
We note that \cite{ConnellyWhiteley1996} used the concept of an energy function as motivation for their definition of prestress stability. Their definition of a \textit{tensegrity framework} uses inequalities on edge lengths to distinguish bars from cables and struts. Our definition puts the energy function at center stage and also allows for a space of control parameters $\Omega$, which we need in order to define catastrophe discriminants $\cdisc \subset \Omega$ below.
\end{remark}

\begin{definition}\label{definition:stable-elastic-tensegrity-framework}
We describe the interaction between an elastic tensegrity framework and the energy function given in (\ref{equation:energy-fcn-Q-with-square-roots}) with the following definitions.
\begin{enumerate}
    \item Fix a tuple of control parameters $y \in Y$. An elastic tensegrity framework in configuration $(x,y)$ is \textit{stable} if the internal parameters $x \in X$ are a \textit{strict local minimum} of the energy function $Q$ restricted to the {algebraic set $\{ x \in X : b(x,y)=0 \}$} 
    of internal parameters satisfying the bar constraints $b(x,y)=0$ of (\ref{equation:bar-constraints-bij}).
    \item {For fixed controls $y \in Y$ we collect all strict local minima in the \textit{stability set} $\mathcal{S}_y$, defined as all internal parameters $x \in X$ such that the corresponding elastic tensegrity framework $(x,y)$ is stable.}
    \item The \textit{stability correspondence} $\mathcal{SC}$ is the set of pairs $(x,y) \in X \times Y$ such that $x \in \mathcal{S}_y$. For a given subset $\Omega \subset Y$ of controls we let $\mathcal{SC}_\Omega$ be all $(x,y) \in X \times \Omega \subset X \times Y$ such that  $x \in \mathcal{S}_y$.
\end{enumerate}
If we are only interested in a subset of control parameters $\Omega \subset Y$ these definitions apply verbatim with $\Omega$ replacing $Y$.
\end{definition}

\begin{example}[Zeeman's catastrophe machine (continued)]\label{ex:zeeman-3}
We continue with Example \ref{ex:zeeman-2}. Figure \ref{figure:zeemans-machine-stability-set} shows Zeeman's catastrophe machine in a stable configuration. However, for that specific value of $y$ the stability set $\mathcal{S}_y$ contains two points, with the second configuration shown in grey. Since the constraints $b(x,y)=0$ essentially describe a circle we can also plot the periodic energy function in Figure \ref{figure:zeemans-machine-stability-set}. For the particular value of the controls $y \in \Omega$ we chose, there are two local minima, and hence $|\mathcal{S}_y| = 2$.
\end{example}
\begin{figure}[!htb]
    \centering
    \includegraphics[width=0.7\textwidth]{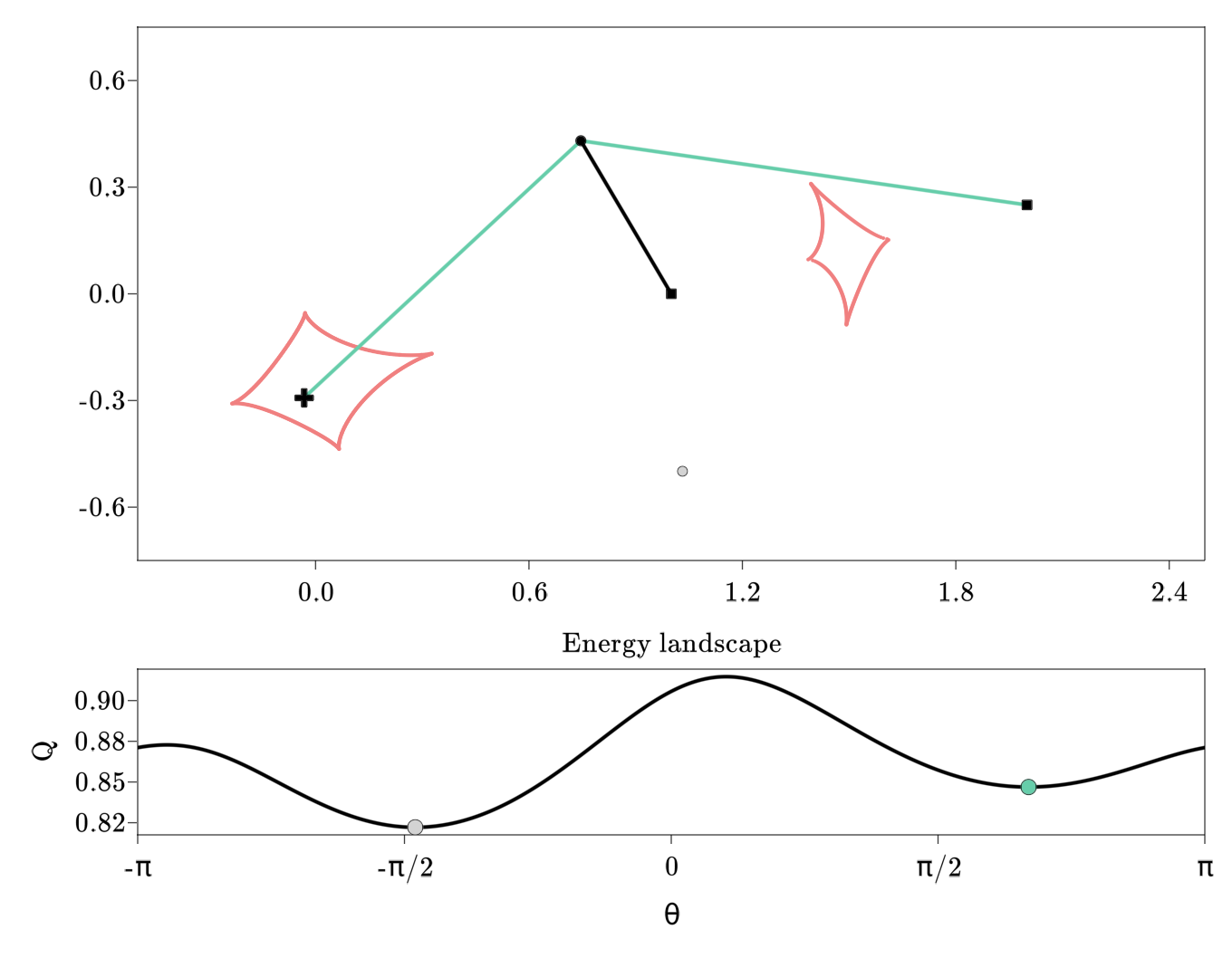}
    \caption{Zeeman machine in a stable configuration with $|\mathcal{S}_y| = 2$. The second stable position of node 4 is depicted in gray.}
    \label{figure:zeemans-machine-stability-set}
\end{figure}
In the following, we focus on stable elastic tensegrity frameworks and the behavior when control parameters $y \in \Omega \subset Y$ change.
For this, consider a \textit{smooth path of control parameters}
\begin{align}\label{eqn:smooth-path-control-parameters}
    y:[0,1] \subset \mathbb{R} &\to \Omega \subset Y\\
    t &\mapsto y(t) \nonumber
\end{align}
and an \textit{initial condition} $(x(0),y(0))$ which is stable according to Definition \ref{definition:stable-elastic-tensegrity-framework}. We are interested in the time evolution of the internal parameters $x(t)$ determined by minimizing $Q$ constrained by $b$ for the given path $y(t)$ of control parameters. In particular, can we identify certain regions where small changes in $y(t)$ can cause large changes in the tensegrity framework? In Section \ref{section:algebraic-from-the-start} we solve an algebraic reformulation of this problem, defining the \textit{catastrophe discriminant} $\cdisc \subset \Omega \subset Y$ and proving that  as long as our controls $y(t)$ avoid a smaller, semialgebraic \textit{catastrophe set} $\cset \subset \cdisc$, then stable local minima at the initial condition are preserved, evolving as a unique path of stable local minima $(x(t),y(t)) \in \mathcal{SC}$.

\section{Algebraic reformulation}\label{section:algebraic-from-the-start}
In this section, we transfer questions about the stability of elastic tensegrity frameworks into an algebraic problem.
The motivation is as follows. The computation of $\mathcal{S}_y$ is in general a very hard problem since the points in $\mathcal{S}_y$ are \textit{all} local minima of a constrained optimization problem. Thus, standard optimization methods are not sufficient since they yield in each run at most one local minimum and cannot provide guarantees to find all local minima.
In contrast, if we work with systems of polynomial equations we can
apply tools from numerical nonlinear algebra to obtain all solutions. This is discussed in more detail in Section \ref{section:algebraic-computations}.

In the following, let $([n],E)$ be an elastic tensegrity framework with variables $p,\ell,r,c$ from (\ref{definition:variable-names}) partitioned into the internal parameters $x \in X \simeq \C^{m_1}$ and the control parameters $y \in Y \simeq \C^{m_2}$.
Compared to the previous section we now work over the complex numbers and the real numbers both. This will allow us to use homotopy continuation in Section \ref{section:algebraic-computations} to find all critical points of the optimization problem. It also allows us to use the degree of an algebraic set. If $A$ is any subset of some $\C^m$ then let $A_\R$ denote the points with real-valued coordinates. Finally, we will sometimes refer to \emph{general} $a \in A$, by which we mean that $a$ lies outside some algebraic subset of $A$, i.e., $a$ lies in some Zariski-open subset of $A$.

Let $\Omega$ be a smooth algebraic subset of the control parameters $Y$ we wish to manipulate with controls $y(t) \in \Omega$. This allows us, among other things, to consider movement of a node constrained to motion in a sphere, perhaps determined by a rigid bar.
We introduce variables $\delta_{ij}$ for $ij \in C$ to eliminate the square roots in the potential energies $q_{ij}$. For $ij \in E$, let
\begin{align*}
g_{ij} =
\begin{cases}
\ell_{ij}^2 - \sum_{k \in [d]} (p_{ik} - p_{jk})^2 \hspace{0.5cm} \text{if } ij \in B\\
\delta_{ij}^2 - \sum_{k \in [d]} (p_{ik} - p_{jk})^2 \hspace{0.5cm} \text{if } ij \in C
\end{cases}
\end{align*}
and denote by $g: X \times \C^{|C|} \times Y \to \C^{|E|}$ the polynomial system whose component functions are the $g_{ij}$.
Furthermore, denote by $\mathcal{G}_y$ the zero set of $g$ for a fixed $y \in Y$
$$\mathcal{G}_y := \{ (x, \delta) \in X \times \C^{|C|} \,\,\,\, | \,\,\,\, g(x, \delta, y) = 0 \}\,.$$
For $ij \in C$ let
\begin{equation*}
    \widetilde{q_{ij}} = \frac{1}{2} c_{ij} (\delta_{ij} - r_{ij})^2 \hspace{1cm} \text{with } \hspace{1cm} \widetilde{Q}_y = \sum_{ij \in C} \widetilde{q_{ij}}
\end{equation*}
an algebraic energy function $\widetilde{Q}_y$. The subscript emphasizes possible dependency~on~$y \in Y$.

To study the stability set $\mathcal{S}_y$ we look at the critical points of $\widetilde{Q}_y(x, \delta)$ subject to $(x, \delta) \in (\mathcal{G}_y)_\R$. 
A point $(x, \delta) \in (\mathcal{G}_y)_\R$ is a critical point of the energy function $\widetilde{Q}_y$ if the gradient $\nabla \widetilde{Q}_y$ is orthogonal to the tangent space of $(\mathcal{G}_y)_\R$ at $(x, \delta)$.
If the algebraic set $\mathcal{G}_y$ is a complete intersection, i.e., the codimension of $\mathcal{G}_y$ equals $|E|$, then
we can directly apply the technique of Lagrange multipliers to compute the critical points. 
In the following, we assume for ease of exposition that this is the case.
However, this is not a critical assumption and the results can be extended to non-complete intersection by using standard numerical nonlinear algebra techniques for randomizing overdetermined systems (see \cite[Chapter 13]{Sommese:Wampler:2005}).

Throughout the article we will assume that $\mathcal{G}_y$ is smooth, since at non-manifold points the first-order conditions for critical points are not applicable. This is really a requirement on the choice of $y \in \Omega \subset Y$. In particular, $\mathcal{G}_y$ should be nonempty and smooth for all $y \in \Omega$. Singular configuration spaces of the underlying bar and joint framework would introduce even more difficulties than we address here, and the behavior of local minima would be highly complicated. The disappearance of local minima on smooth configuration spaces is already difficult and interesting, and this is our topic. Both Zeeman's catastrophe machine and the larger example we discuss in Section \ref{section:elastic-four-bar} have $\mathcal{G}_y$ smooth for every $y \in \Omega$, but still display interesting catastrophic behavior which can be predicted or avoided using our methods.

We introduce the variables $\lambda_{ij}$ for $ij \in E$ to act as \textit{Lagrange multipliers} and~let
\begin{equation}\label{equation:Lagrangian-algebraic}
    L_y(x,\delta,\lambda) = \widetilde{Q}_y + \sum_{ij \in E} \lambda_{ij} g_{ij} \,.
\end{equation}

\begin{definition}\label{definition:dL-and-its-zero-set-LC}
Define the polynomial map $dL_y$ by letting its component functions be the various partial derivatives of $L_y$ with respect to $x, \delta$ and $\lambda$.
\begin{equation*}
    dL_y := \dfrac{\partial L_y}{\partial(x,\delta,\lambda)} : X \times \C^{|C|} \times \C^{|E|} \to X \times \C^{|C|} \times \C^{|E|},\;
    (x,\delta,\lambda) \mapsto dL_y\big(x,\delta,\lambda\big).
\end{equation*}
Denote the algebraic sets $\mathcal{L}_y :=  dL_y^{-1}(0) \subset X \times \C^{|C|} \times \C^{|E|} $ and
\begin{equation*}
    \mathcal{LC} := \{ (x,\delta, \lambda, y) \, | \, (x,\delta, \lambda) \in \mathcal{L}_y \} \subset X \times \C^{|C|} \times \C^{|E|} \times \Omega
\end{equation*}
and let $\mathcal{LC}_{reg}$ denote its open dense subset of smooth points and $\mathcal{LC}_{sing}$ its singular locus.
\end{definition}

\begin{proposition}\label{prop:finite-cardinality-Ly}
If the dimension of $\Omega$ and $\mathcal{LC}$ coincide, then for general $y \in \Omega$ the variety $\mathcal{L}_y$ is finite and has the same cardinality $\mathcal{N}$.
For all $y \in \Omega$ the variety $\mathcal{L}_y$ contains at most $\mathcal{N}$ isolated points.
\end{proposition}
\begin{proof}
This is a standard result in algebraic geometry, e.g., \cite[Theorem 7.1.6]{Sommese:Wampler:2005}.
\end{proof}

\begin{definition}\label{definition:equilibrium-degree}
Given $\Omega \subset Y$ we define the \textit{equilibrium degree of a framework} to be the cardinality of $\mathcal{L}_y$ for general $y \in \Omega$. Proposition \ref{prop:finite-cardinality-Ly} implies that the equilibrium degree is well-defined.
\end{definition}

\begin{example}[Zeeman's catastrophe machine (continued)]\label{ex:zeeman-4}
We continue our running example with edges $E = \{14,24,34\}$ partitioned as $B=\{14\}$ and $C=\{24,34\}$.
Recall from Example \ref{ex:zeeman-2} we had $\Omega_\R = \left\{ \,\, (0,0,2,-1,p_{31},p_{32},1,1,1,0.5, 0.5) \,\, : \,\, (p_{31},p_{32}) \in \mathbb{R}^2 \right\} \subset Y$, and $X_\R = \left\{ \, (p_{41}, p_{42}) \, \right\} \simeq \mathbb{R}^2$. We write $x=(p_{41},p_{42})$.
The polynomials defining our constraints~are
\begin{equation*}
    g(x,\delta,y) = \left[ \begin{array}{c}
        1^{2} - {\left(0 - p_{41}\right)}^{2} - {\left(0 - p_{42}\right)}^{2}\\
        \delta_{24}^{2} - {\left(2 - p_{41}\right)}^{2} - {\left(-1 - p_{42}\right)}^{2}\\
        \delta_{34}^{2} - {\left(p_{31} - p_{41}\right)}^{2} - {\left(p_{32} - p_{42}\right)}^{2}
    \end{array} \right] = \left[ \begin{array}{c}
        0 \\
        0 \\
        0
    \end{array} \right].
\end{equation*}
We obtain the Lagrangian of (\ref{equation:Lagrangian-algebraic}) as
\begin{multline*}
    L_{(p_{31},p_{32})} =
    \frac{1}{4} {\left(\delta_{24} - 1\right)}^{2} + 
    \frac{1}{4} {\left(\delta_{34} - 1\right)}^{2} +
    {\left(1 - {p_{41}}^{2} - {p_{42}}^{2} \right)} \lambda_{14} + \\
    {\left(  \delta_{24}^{2} - {\left(2 - p_{41}\right)}^{2} - {\left(-1 - p_{42}\right)}^{2} \right)} \lambda_{24} +
    {\left(  \delta_{34}^{2} - {\left(p_{31} - p_{41}\right)}^{2} - {\left(p_{32} - p_{42}\right)}^{2}  \right)} \lambda_{34}.
\end{multline*}
The polynomial system $dL_y$ is given by
\begin{equation*}
    dL_{(p_{31},p_{32})}(x,\delta,\lambda) = \left[ \begin{array}{c}
        -2 \, \lambda_{14} {p_{41}} - 2 \, \lambda_{24} {\left(2 - p_{41}\right)} - 2 \, \lambda_{34} {\left(p_{31} - p_{41}\right)}\\
        -2 \, \lambda_{14} {p_{42}} - 2 \, \lambda_{24} {\left(-1 - p_{42}\right)} - 2 \, \lambda_{34} {\left(p_{32} - p_{42}\right)}\\
        \frac12 {\left(\delta_{24} - 1\right)} + 2 \, \delta_{24} \lambda_{24}\\
        \frac12 {\left(\delta_{34} - 1\right)} + 2 \, \delta_{34} \lambda_{34}\\
        1 - {p_{41}}^{2} - {p_{42}}^{2}\\
        \delta_{24}^{2} - {\left(2 - p_{41}\right)}^{2} - {\left(-1 - p_{42}\right)}^{2}\\
        \delta_{34}^{2} - {\left(p_{31} - p_{41}\right)}^{2} - {\left(p_{32} - p_{42}\right)}^{2}
    \end{array} \right].
\end{equation*}
The equilibrium degree for this framework is $16$. This means that for generic $(p_{31},p_{32}) \in \Omega$ the equations $dL_{(p_{31},p_{32})}(x,\delta,\lambda) = 0$ have $16$ isolated solutions over the complex numbers.
\end{example}

We have particular interest in those parameter values $y \in \Omega$ where the number of regular isolated solutions $|\mathcal{L}_y|$ of $dL_y(x,\delta,\lambda) = 0$ is less than the equilibrium degree of the framework.
\begin{definition}\label{definition:catastrophe-hypersurface}
Define the \textit{catastrophe discriminant} $\cdisc \subset \Omega$ as the Zariski closure of the set of critical values of the projection map
\begin{equation*}
    \pi: \mathcal{LC} \to \Omega, \hspace{1cm} z = (x,\delta,\lambda,y) \mapsto y = \pi(z).
\end{equation*}
By \textit{critical values} we mean those $\pi(z) \in \Omega$ such that
there exists a nonzero tangent vector $v \in T_z \mathcal{LC}$ in the kernel of the linear map $d\pi_z$.
The catastrophe discriminant is an algebraic subset of $\Omega$ of codimension 1.
\end{definition}

\begin{definition}\label{definition:catastrophe-degree}
The \textit{catastrophe degree} of an elastic tensegrity framework is the degree of the algebraic set $\cdisc$, i.e. the number of complex-valued points in the intersection of $\cdisc$ with a general linear space of complementary dimension.
\end{definition}

\begin{example}[Zeeman's catastrophe machine (continued)]\label{ex:zeeman-5} We continue with Example \ref{ex:zeeman-4}. Refer back to Figure \ref{figure:computed-catastrophe-curves} which shows the catastrophe discriminant $\cdisc \cap \Omega_\R$ for Zeeman's machine with controls $\Omega_\R$ defined in Example \ref{ex:zeeman-2}. Note that $\cdisc$ does not depend on $y$ but just on the choice of $X$ and $\Omega \subset Y$.
Here, $\cdisc$ is an algebraic plane curve of degree $72$. That is, the catastrophe degree is $72$. Over the finite field $\mathbb{Z}_{65521}$ the catastrophe discriminant $\cdisc$ is the zero set of the 2701-term polynomial
\begin{multline*}
p_{31}^{72}+13109\,p_{31}^{71}p_{32}-13055\,
      p_{31}^{70}p_{32}^{2}+10676\,p_{31}^{69}p_{32}^{3}+
      7407\,p_{31}^{68}p_{32}^{4}+4476\,p_{31}^{67}p_{32}^{5}+31981\,p_{31}^{66}p_{32}^{6} + \\
      12338\,p_{31}^{65}p_{32}^{7}-8796\,p_{31}^{64}p_{32}^{8}+19319\,p_{31   }^{63}p_{32}^{9}+4482\,p_{31}^{62}p_{32}^{10}+\ldots -709\,p_{31}-32406\,p_{32}+540 .
\end{multline*}
Figure \ref{figure:computed-catastrophe-curves} also shows the catastrophe set $\cset$ which we define below. As we move controls $y(t) \in \Omega_\R$ the set $\cset$ detects possible changes in the number of local minima, and hence possible catastrophe. This was the largest catastrophe discriminant we could compute using symbolic methods, and we had to replace $\mathbb{Q}$ by a finite field for the computation to terminate. See Section \ref{section:algebraic-computations} for a bit more discussion of this example. In contrast, the homotopy continuation methods we discuss in Section \ref{section:algebraic-computations} can handle much larger examples. 
\end{example}

\begin{definition}\label{definition:catastrophe-set-D-Omega}
 Using the map $\pi$ of Definition \ref{definition:catastrophe-hypersurface} we define 
$$\cset := \{ y \in \cdisc \cap \Omega_\R \, | \text{ there exists } (x, \delta, \lambda, y) \in \left( \pi^{-1}(y) \right)_\R \textnormal{ with $\delta \ge 0$} \, \} \subset \cdisc \cap \Omega_\R $$
to be the \textit{catastrophe set}. This is the part of the catastrophe discriminant $\cdisc$ that relates to the original problem.
\end{definition}

We note that the  \textit{catastrophe set} $\cset$ partitions $\Omega_\R$ into cells within which the
number of strict local minima is constant. Figure \ref{figure:computed-catastrophe-curves} depicts the number $|\mathcal{S}_y|$ of stable local minima for a typical point $y$ in each connected component of the complement $ \Omega_\mathbb{R} \setminus \cset$. Look ahead to Figure \ref{fig:fourbar_catastrophes} for another illustration of this phenomenon for the elastic four-bar linkage discussed in Section \ref{section:elastic-four-bar}.

\begin{proposition}\label{prop:D-omega-semialgebraic}
The \textit{catastrophe set} $\cset$ is a semialgebraic set.
\end{proposition}
\begin{proof}
 $\cset$ is the projection of a semialgebraic set and hence again semialgebraic by Tarski-Seidenberg. 
\end{proof}

We now begin to prove theorems justifying our interest in $\cdisc$ and $\cset$.  Our goal is to prove Theorem \ref{theorem:path-lifting-correct}, which says that controls $y(t)$ avoiding the semialgebraic catastrophe set $\cset$ always correspond to stable local minima, and thus avoid catastrophes where local minima disappear discontinuously. This is called \textit{catastrophe} since a real-world system would be forced to move rapidly towards the nearest remaining local minima, and since without knowledge of $\cset$ this sudden change in behavior would be very surprising.

For the remainder of this section we work mostly with real algebraic sets, since our goal is connecting back to the original problem. Since all the complex algebraic sets were defined by polynomials with real coefficients, the same polynomials define real algebraic sets which can be identified with subsets of the complex sets whose coordinates have zero imaginary part.
In particular, at smooth points of a complex algebraic set, the tangent space is equal to the kernel of the Jacobian of the defining polynomials, which has full rank if we have a complete intersection, or if we have applied standard methods of randomizing overdetermined systems in numerical algebraic geometry. Since the defining polynomials have real coefficients, so do their partial derivatives, so that evaluating the Jacobian at a real point yields a matrix with real entries.
A basis for the kernel of such a matrix can always be chosen with vectors whose coordinates are also real-valued. In what follows, if $A \subset \C^m$ then $T_a A_\R$ denotes such a real tangent space at the real point $a \in A_\R$. In particular, if $a \in A_\R$ is a smooth point of the complex algebraic set $A$ then $\text{dim}_\C T_a A = \text{dim}_\R T_a A_\R$ and the set $A_\R$ is a manifold near $a$.

\begin{lemma}\label{lemma:singular-implies-critical}
If $z = (x,\delta,\lambda,y) \in \mathcal{LC}_{\text{sing}}$ then $\pi(z) \in \mathcal{D}_\Omega$.
\end{lemma}

\begin{proof}
Since $\Omega$ is a smooth algebraic set $\Omega = \phi^{-1}(0)$ for some polynomial map $\phi:Y \to \C^k$ not depending on $(x,\delta,\lambda)$. Then $z \in \mathcal{LC}_{\text{sing}}$ implies a rank drop in the Jacobian matrix $dF$ of the polynomial map $F = ( dL_y, \phi)$. But $dF$ has a block structure with top rows $[ A | B ]$ and bottom rows $[ 0 | C ]$ where $A$ is the square Jacobian of $dL_y$ at $(x,\delta,\lambda)$, $B$ is unimportant, and $C$ is the Jacobian of $\phi$ at $y$. Since $\Omega$ is smooth, the bottom rows have full rank and since $z \in \mathcal{LC}_{\text{sing}}$ there exists $(w,0) \neq 0$ with $(w,0)^T dF = 0$, which implies $w^T A = 0$ and hence the square matrix $A$ drops rank. Thus there is $v \neq 0$ with $A v = 0$ and so $(v,0) \in T_z \mathcal{LC}$ with $d\pi_z (v,0) = 0$, which completes the proof.
\end{proof}

\begin{lemma}\label{lemma:singular-implies-critical-real}
If $z \in \mathcal{LC}_\R$ is singular and $\delta \geq 0$ then $\pi(z) \in \cset$.
\end{lemma}
\begin{proof}
By Lemma \ref{lemma:singular-implies-critical}, $\pi(z) \in \cdisc$. But $z \in \mathcal{LC}_\R$ implies all of $x,\delta,\lambda,y$ are real-valued, and with $\delta \geq 0$ we have that $\pi(z)$ satisfies the requirements for belonging to $\cset$.
\end{proof}

\begin{lemma}\label{lemma:dimension-conditions}
Let $y:[0,1] \to (\Omega_\R \setminus \cset)$ avoid $\cset$. If the smooth curve $z:[0,1] \to \mathcal{LC}_\R$ satisfies $\pi(z(t))=y(t)$ and $\delta(t) \geq 0$ then $\text{dim } T_{z(t)} \mathcal{LC}_\R = \text{dim } T_{y(t)} \Omega_\R$
holds for all $t$ if it holds for some $t$.
\end{lemma}
\begin{proof}
Since $\Omega$ is assumed smooth then $\Omega_\R$ is smooth and $t \mapsto y(t)$ stays within one path-connected component, hence $\text{dim } T_{y(t)} \Omega_\R$ is constant for all $t$. Since the path $t \mapsto z(t)$ stays in one path-connected component of $\mathcal{LC}_\R$, the dimension may only change if some $z(t_*)$ is singular. However, since $z(t_*)$ is real-valued and $\delta(t_*) \geq 0$ then Lemma \ref{lemma:singular-implies-critical-real} implies that $\pi(z(t_*)) = y(t_*) \in \cset$, contradicting our assumption.
\end{proof}

We now introduce conditions that will eventually correspond to stability of the elastic tensegrity framework.
For this, we first define a \emph{nondegenerate} point $z \in \mathcal{LC}_\R$ using second-order sufficient conditions for local minima of nonlinear constrained optimization problems. 

\begin{definition}\label{definition:strict-local-minimum-positive-definite-projected-Hessian} 
Let $d^2 \widetilde{Q}_y$ and $d^2 g_{ij}$ be the Hessian matrices of $\widetilde{Q}_y$ and $g_{ij}$ respectively, when viewed as real-valued functions of the real variables $x$ and $\delta$, with $d^2 \widetilde{Q}(x,\delta)$ and $d^2 g_{ij}(x,\delta)$ denoting their evaluation at the point $(x,\delta)$. Let $dg$ denote the Jacobian of the constraints $g$ viewed again as real-valued functions of the real variables $x$ and $\delta$, with $dg(x,\delta)$ its evaluation at a point. Let
\begin{equation}\label{equation:hessian-of-lagrangian}
    H = d^2 \widetilde{Q}_y(x,\delta) + \sum_{ij \in E} \lambda_{ij} d^2 g_{ij}(x,\delta).
\end{equation}
We say that $z = (x,\delta,\lambda,y) \in \mathcal{LC}_\R$ is \emph{nondegenerate} if 
there is a positive definite matrix
\begin{equation}\label{equation:proj-hessian}
    W(z) := V^T H V,
\end{equation}
where $V$ has columns a real orthonormal basis of the kernel of $dg(x,\delta)$. The matrix $W$ is called the projected Hessian and, with our setup, its being positive definite guarantees that $(x,\delta)$ is a strict local minimum for $\widetilde{Q}_y$ restricted to $(\mathcal{G}_y)_\R$. See, e.g., \cite[page 81]{Gill:Murray:Wright:1982}.
\end{definition}

\begin{lemma}\label{lemma:proj-hessian-zero}
Let $z \in \mathcal{LC}_\R$ have singular $W(z)$ and $\delta \geq 0$. Then $\pi(z) \in \cset$.
\end{lemma}
\begin{proof}
First note that $W(z)$ is only defined up to a choice of orthonormal basis in $V$, but the property of being singular is invariant under such changes. Since $W(z)$ is singular,
there exists $u \neq 0$ with $V^T H V u = 0$.
Placing parentheses $V^T \Big( H V u \Big) = 0$ we see that $H V u$ is in the normal space of $(\mathcal{G}_{y})_\R$ at $(x,\delta)$. But then there must exist a linear combination $w$ writing $H V u$ in terms of the columns of $dg(x,\delta)^T$, and hence $(V u, -w) \in \text{ker } d^2 L$ where 
\begin{equation*}
    d^2 L = \left[ \begin{array}{cc}
        H & dg(x,\delta)^T \\
        dg(x,\delta) & 0
    \end{array} \right]
\end{equation*}
is the Hessian of the Lagrangian $L_{y}$ of (\ref{equation:Lagrangian-algebraic}).
Note that the vector $(V u, -w)$ extends to a tangent vector $(Vu, -w, 0)$ of $T_{z} \mathcal{LC}_\R$ by appending zeros in the $\Omega$ components. This tangent vector clearly projects to zero by $d \pi_{z}$. Since $z \in \mathcal{LC}_\R$ and $\delta \geq 0$ this means that $\pi(z) \in \cset$.
\end{proof}

\begin{theorem}\label{theorem:pathlifting-local-minima}
Let $z:[0,1] \to \mathcal{LC}_\R$ with components $z(t) = (x(t),\delta(t),\lambda(t),y(t))$ be a smooth curve of smooth points in $\mathcal{LC}_\R$ with $\delta(t) \geq 0$ and $\pi(z(t)) \notin \cset$ for all $t$. If the initial point $z(0)$ is nondegenerate then $(x(t),\delta(t))$ are strict local minima for $\widetilde{Q}_{y(t)}$ on $(\mathcal{G}_{y(t)})_\R$ for all $t$.
\end{theorem}

\begin{proof}
Since $t \mapsto z(t)$ is a smooth curve of smooth points we can find a smooth curve $t \mapsto V(t)$ of matrices whose columns form a basis for $\text{ker } dg(x(t),\delta(t))$ at each $t$, and an associated smooth curve of matrices $t \mapsto H(t)$ by using $x(t),\delta(t),\lambda(t),y(t)$ in formula (\ref{equation:hessian-of-lagrangian}) above.
Since $z(0)$ is assumed nondegenerate, $V(0)^T H(0) V(0)$ has all positive eigenvalues.
As $t$ varies smoothly, so do the real eigenvalues of the symmetric matrices $V(t)^T H(t) V(t)$. Suppose at some $t$ there appears a zero eigenvalue in $V(t)^T H(t) V(t)$.
Then Lemma \ref{lemma:proj-hessian-zero} implies that $\pi(z(t)) \in \cset$, a contradiction.
Therefore $V(t)^T H(t) V(t)$ remains positive definite for all $t$, which by the sufficient conditions for strict local minima \cite[page 81]{Gill:Murray:Wright:1982} completes the proof.
\end{proof}

We now discuss how our algebraic reformulation relates back to the original problem.
In our algebraic reformulation we removed the square roots by introducing the additional variables $\delta_{ij}$ for $ij \in C$.
In the following Lemma we assume that all elastic cables are in tension since such systems are only structurally stable when self-stress is induced.

\begin{lemma}\label{lemma:relate-critical-back}
Consider an elastic tensegrity framework in stable configuration $(x,y) \in \mathcal{SC}$.
Define $\delta_{ij} = \sqrt{ \sum_{k \in [d]} (p_{ik} - p_{jk})^2}$ and let
\begin{equation}\label{equation:cables-in-tension}
    \delta_{ij} > r_{ij} > 0
\end{equation}
for every $ij \in C$, so that all elastic cables are in tension.
Then there exists $\lambda \in \R^{|E|}$ such that $(x,\delta,\lambda,y) \in \mathcal{LC}_\R$. 
\end{lemma}
\begin{proof} 
Let $V_{b,y} := \{ x \in X : b(x,y) = 0 \}$. Now consider the map $s_y:X \to \mathbb{R}^{|C|}$ defined by coordinate functions $(s_y)_{ij}(x) =  \sqrt{\sum_{k \in [d]} (p_{ik} - p_{jk})^2} $. Restricting this map to $V_{b,y}$ we have a local diffeomorphism between $V_{b,y}$ and its graph
\begin{equation*}
    \{ \, (x, \, s_y(x) ) \, : x \in V_{b,y} \} \subset X \times \mathbb{R}^{|C|}
\end{equation*}
which provides a local diffeomorphism between $V_{b,y}$ and $(\mathcal{G}_y)_\R$ near any point $x \in V_{b,y}$ satisfying (\ref{equation:cables-in-tension}). Observe that, by construction, $\widetilde{Q}_y$ takes values on the image points equal to the values taken by $Q$ on the domain $V_{b,y}$, provided condition (\ref{equation:cables-in-tension}) holds. Therefore, $x \in V_{b,y}$ is a strict local minimum of $Q$ on $V_{b,y}$ if and only if $(x,s_y(x))$ is a strict local minimum of $\widetilde{Q}_y$ on $(\mathcal{G}_y)_\R$.
Hence, by first-order necessary conditions for local extrema we know that there exists $\lambda$ such that $(x,s_y(x),\lambda,y) \in \mathcal{LC}_\R$, completing the proof. 
\end{proof}

Finally, we are able to prove that the stability of the corresponding elastic tensegrity framework is preserved by avoiding only $\cset \subset \Omega_\R$.

\begin{theorem}\label{theorem:path-lifting-correct}
If $z(0) \in \mathcal{LC}_\R$ is nondegenerate, $\text{dim } T_{z(0)} \mathcal{LC}_\R = \text{dim } T_{y(0)} \Omega_\R$, and $y:[0,1] \to (\Omega_\R \setminus \cset)$ is a smooth path of controls, then there exists a unique smooth map $z:[0,1] \to \mathcal{LC}_\R$ with $\pi(z(t)) = y(t)$ and $(x(t),y(t)) \in \mathcal{SC}$ for all $t \in [0,1]$, provided that condition (\ref{equation:cables-in-tension}) holds for all $t \in [0,1]$.
\end{theorem}

\begin{proof}
Since $z(0)$ is nondegenerate we claim that $d( \pi|_{\mathcal{LC}_\R} )_{z(0)}: T_{z(0)} \mathcal{LC}_\R \to T_{y(0)} \Omega_\R$ is an isomorphism. First, note that $z(0)$ is a smooth point of $\mathcal{LC}_\R$ by Lemma \ref{lemma:singular-implies-critical-real} and the assumption that $y(t) \notin \cset$.
Thus $d( \pi|_{\mathcal{LC}_\R} )_{z(0)}$ is an isomorphism unless there is $v \neq 0$ with $v \in T_{z(0)} \mathcal{LC}_\R$ and $d\pi_{z(0)}(v) = 0$. The second condition implies that $v = (v_1, v_2, 0)$ where at least one of $v_i \neq 0$.
But then $(v_1,v_2,0) \in T_{z(0)} \mathcal{LC}_\R$ implies $H v_1 + dg^T v_2 = 0$ and $dg \,\, v_1 = 0$. If $v_1 = 0$, then $dg^T v_2 = 0$ implies $v_2 = 0$ since $dg$ is surjective. Thus $v\neq 0$ implies $v_1 \neq 0$ so that $v_1 = V u$ for some $u \neq 0$. Then $Hv_1 + dg^T v_2 = 0$ implies $HVu + dg^T v_2 = 0$. But $V^T dg^T = 0$ so that $V^T H V u = 0$ for $u \neq 0$, contradicting the nondegeneracy of $z(0)$.
Therefore $d( \pi|_{\mathcal{LC}_\R} )_{z(0)}$ is an isomorphism and the inverse function theorem implies that $\pi$ is a local diffeomorphism at $z(0)$.

Let $\mathcal{U}$ be the  largest  neighborhood of $y(0)$ in $\Omega_\R$ such that there is a neighborhood $\mathcal{V}$ of $z(0)$ in $\mathcal{LC}_\R$ with $\pi|_\mathcal{V}:\mathcal{V} \to \mathcal{U}$ a diffeomorphism.  Assume for now that $\mathcal{U}$ exists. 
We then define $z(t) := \pi|_\mathcal{V}^{-1} (y(t))$ for all $t$ with $y(t) \in \mathcal{U}$.
With $z(t)$ defined we check if $\delta_{ij}(t) > r_{ij}(t) > 0$ holds.
By Lemma \ref{lemma:singular-implies-critical-real} and $y(t) \notin \cset$ all such $z(t)$ are smooth points of $\mathcal{LC}_\R$. By Theorem \ref{theorem:pathlifting-local-minima} we conclude that $(x(t),\delta(t))$ are strict local minima for $\widetilde{Q}_{y(t)}$ on $\mathcal{G}_{y(t)}$ for all such $t$.
Using the maps $s_{y(t)}$ from the proof of Lemma \ref{lemma:relate-critical-back} we have local diffeomorphisms mapping the strict local minima for $\widetilde{Q}_{y(t)}$ on $\mathcal{G}_{y(t)}$ to strict local minima of $Q$ on $V_{b,y(t)}$ so that $(x(t),y(t)) \in \mathcal{SC}$ for all such $t$. If $y(1) \in \mathcal{U}$, we are done.

Otherwise there exists $\epsilon > 0$ with $y(t) \in \mathcal{U}$ for all $t \in [0,\epsilon)$ and $y(\epsilon) \notin \mathcal{U}$. Define $z(t)$ as above for all $t \in [0,\epsilon)$ and consider the limit as $t \to \epsilon$. We know that $y(\epsilon) \notin \cset$ by assumption. Let $z(\epsilon) := \text{lim } z(t)$ as $t \to \epsilon$.
Then $z(\epsilon)$ is again real-valued. If $z(\epsilon)$ were singular, then Lemma \ref{lemma:singular-implies-critical-real} implies $\pi(z(\epsilon)) = y(\epsilon) \in \cset$, a contradiction. Thus $z(\epsilon)$ is a smooth point of $\mathcal{LC}_\R$.
Also by the same argument as the proof of Theorem \ref{theorem:pathlifting-local-minima} we know that all $W(z(t))$ for $t < \epsilon$ have positive eigenvalues.
If $W(z(\epsilon)) := \text{lim } W(z(t))$ as $t \to \epsilon$ has a zero eigenvalue, then Lemma \ref{lemma:proj-hessian-zero} implies that $\pi(z(\epsilon)) \in \cset$, again a contradiction. Thus all eigenvalues of $W(z(\epsilon))$ are positive so that $z(\epsilon)$ is nondegenerate, and previous arguments imply that $(x(\epsilon),y(\epsilon)) \in \mathcal{SC}$ as well.
By Lemma \ref{lemma:dimension-conditions} we have $\text{dim } T_{z(\epsilon)} \mathcal{LC}_\R = \text{dim } T_{y(\epsilon)} \Omega_\R$.
Replacing $z(0)$ by $z(\epsilon)$ in our previous argument we can find neighborhoods $\mathcal{U}'$ of $y(\epsilon)$ and $\mathcal{V}'$ of $z(\epsilon)$ with $\pi|_{\mathcal{V}'}:\mathcal{V}' \to \mathcal{U}'$ a diffeomorphism. But then $\mathcal{U} \cup \mathcal{U}'$ is a strictly larger neighborhood of $y(0)$ satisfying the same conditions, contradicting our choice of $\mathcal{U}$. Thus $y(1) \in \mathcal{U}$, as needed.

 To complete the proof, we use Zorn's lemma to show $\mathcal{U}$ exists. Let $\mathcal{A}$ contain all neighborhoods $\mathcal{U}_a$ of $y(0)$ in $\Omega_\R$ such that there exists a neighborhood $\mathcal{V}_a$ of $z(0)$ in $\mathcal{LC}_\R$ which is diffeomorphic to $\mathcal{U}_a$ by restricting $\pi$ to $\mathcal{V}_a$. $\mathcal{A}$ is partially ordered by set inclusion. $\mathcal{A}$ is nonempty since we showed $\pi$ is a local diffeomorphism at $\pi(z(0)) = y(0)$. Let $\mathcal{B}$ be a chain in $\mathcal{A}$, i.e. a subset of $\mathcal{A}$ that is totally ordered. Then the union $\mathcal{C} = \cup \,\, \mathcal{U}_b$ of all the elements of $\mathcal{B}$ is an upper bound for the chain $\mathcal{B}$, since it contains each element of $\mathcal{B}$ and it is an element of $\mathcal{A}$. To see this, consider that $\mathcal{C}$ is a union of open sets containing $y(0)$ and therefore is a neighborhood of $y(0)$. Let $\mathcal{D}$ be the union of all the neighborhoods $\mathcal{V}_b$ of $z(0)$ which exist for each $\mathcal{U}_b \in \mathcal{B}$. Then $\mathcal{D}$ is again a neighborhood of $z(0)$ and the same map $\pi$ was used for the diffeomorphisms by restricting to each of the $\mathcal{V}_b$ and hence also provides the required diffeomorphism between $\mathcal{D}$ and $\mathcal{C}$. Thus every chain has an upper bound, and Zorn's lemma implies there is a maximal element in $\mathcal{A}$. This completes the proof.
\end{proof}

\section{Computations using numerical nonlinear algebra}\label{section:algebraic-computations}
In this section we use the algebraic reformulation developed in the previous section to describe numerical nonlinear algebra routines that can be used to answer the following three computational problems:
\begin{enumerate}
    \item Given $\gamma \in \Omega_{\R}$ compute $\mathcal{S}_\gamma$.
    \item Given an algebraic path $y(t): [0,1] \to \Omega_{\R} \subset Y$ and an initial configuration $x_0 \in \mathcal{S}_{y(0)}$ compute the path points $\gamma \subset y([0,1])$ where a catastrophe might occur (a local minimum disappears).
    \item Given a control set $\Omega$ compute the catastrophe set $\cset$.
\end{enumerate}

We start with the first question.
Recall that $dL_y$ is a polynomial system.
We can compute all isolated solutions of a polynomial system using homotopy continuation methods \cite{Sommese:Wampler:2005}. Homotopy continuation methods work by first solving a compatible but simpler start system and then keeping track of these solutions as the start system is deformed into the system we intended to solve originally (the target system).
For our computations we use the software package \texttt{HomotopyContinuation.jl} \cite{BT2018}.
To compute $\mathcal{S}_{\gamma}$ for a given $\gamma \in \Omega_{\R}$ we therefore first solve $dL_\gamma(x,\delta,\lambda) = 0$ which results in finitely many complex solutions $\mathcal{L}_\gamma$.
Of these complex solutions we then select those solutions whose components are real-valued and then further select those real-valued solutions where the projected Hessian defined in \eqref{equation:proj-hessian} is positive definite.
Note that computing solutions to $dL_\gamma(x,\delta,\lambda)=0$ usually requires that we track many more paths than the equilibrium degree of $\mathcal{L}_\gamma$. If the goal is to compute $\mathcal{S}_y$ for many different $y \in \Omega_{\R}$ it is more efficient to use a \emph{parameter homotopy} \cite{morgan1989coefficient, Sommese:Wampler:2005}.
There, the idea is to first compute $\mathcal{L}_{y_0}$ for a general (complex) $y_0 \in \Omega$ and then to use the parameter homotopy $H(z,t) = dL_{ty_0 + (1-t)y}(z)$ to efficiently compute $\mathcal{L}_{y}$. Using this parameter homotopy approach allows us to track only the minimal numbers of paths that still guarantee all solutions in $\mathcal{S}_y$ are computed correctly.

Consider the second question where we are given an algebraic path $y(t): [0,1] \to \Omega_{\R} \subset Y$ and an initial configuration $x_0 \in \mathcal{S}_{y(0)}$. We want to compute the path points $\gamma \subseteq y([0,1])$ where a catastrophe might occur.
From the results in Section \ref{section:algebraic-from-the-start} it follows that we want to compute the intersection of $\cset$ and $y([0,1])$.
For this, we first compute the intersection $\cdisc \cap \alpha$ where $\alpha \subset \Omega$ is an algebraic curve containing $y([0,1])$.
For simplicity, we assume that we have the general situation that $\alpha \not\subset \cdisc$.
The catastrophe discriminant $\cdisc$ is given by $\pi(H_\Omega^{-1}(0)))$ with $\pi$ from Definition \ref{definition:catastrophe-hypersurface} and
$$
H_\Omega (x,\delta,\lambda,y) = \begin{bmatrix}
    dL_y(x,\delta,\lambda) \\
    \det\, d^2L_y(x,\delta,\lambda) 
\end{bmatrix} \,.
$$
Since the evaluation of a determinant is numerically unstable it is better to instead use the formulation that there exists a $v \in \mathbb{P}^n$ such that $d^2L_y(x,\delta,\lambda) \cdot v = 0$.
Consider the collection $\{H_\Omega, \pi, \pi^{-1}(\alpha), \mathcal{W}\}$ where $H_\Omega$ and $\pi$ are the polynomial maps defined above and $ \mathcal{W} = \pi^{-1}(\alpha) \cap H_\Omega^{-1}(0)$ contains finitely many solution points. In the case that $\alpha$ is a line this is known as a \emph{pseudo-witness set} \cite{HAUENSTEIN20103349} since it allows us to perform computations on $\cdisc$ without knowing its defining polynomials explicitly.
Since $\mathcal{W}$ is the zero set of a polynomial system it can again be computed by using homotopy continuation techniques.
If $\alpha$ is a line then cardinality of $\mathcal{W}$ is the catastrophe degree of the tensegrity framework.
To compute $\cset \cap y([0,1])$ given $\mathcal{W}$ we have to select from $(x, \delta, \lambda, \gamma) \in \mathcal{W}$
all those solutions which have real-valued coordinates, $\delta > 0$, and $\gamma \in y([0,1]) \subseteq \alpha$. 

We move to the third question and discuss the computation of the catastrophe set $\cset$.
This is more involved since $\cset$ is a positive-dimensional set and we have to decide what ``compute'' means in our context.
Since $\cset$ is a semialgebraic set it can theoretically be defined by a union of finite lists of polynomial equalities and inequalities.
However, computing the describing polynomials is a very challenging computational problem since it requires Gr\"obner bases computations which have exponential complexity. We were able to compute the polynomial defining $\cdisc$ in Example \ref{ex:zeeman-5}, but only over a finite field, and larger examples will likely fail to terminate.
Instead we opt to obtain a sufficiently dense point sample of $\cset$.
The idea is to apply the previously described technique to compute repeatedly the intersection of $\cset$ and a real line $\ell \subset \Omega_\R$.
To proceed we first compute a pseudo-witness set $\{H_\Omega, \pi, \pi^{-1}(\ell_0), \mathcal{W}_0\}$
for a general (complex) line $\ell_0 \subset \Omega$ and then we can compute the
pseudo witness set  $\{H_\Omega, \pi, \pi^{-1}(\ell),\mathcal{W}\}$ by utilizing a parameter homotopy. As discussed above, this is much more efficient for the repeated solution of our equations.
Note that even if the real lines $\ell \subset \Omega_\R$ are sampled uniformly, 
this does not guarantee that the obtained sample points converge to a uniform sample of $\cset$. If uniform sampling is of interest, the procedure can be augmented with a rejection step as described in \cite{breiding2020random}.
The outlined procedure is an effective method to sample the catastrophe discriminant $\cdisc$.
Figure \ref{figure:computed-catastrophe-curves} depicts the point samples obtained for Zeeman's catastrophe machine using this~method, while Figure \ref{fig:fourbar_catastrophes} depicts those obtained for the elastic four-bar framework of Section \ref{section:elastic-four-bar}.

\section{Example: Elastic four-bar framework}\label{section:elastic-four-bar}

We want to demonstrate the developed techniques on another example.
For this we consider a planar four-bar linkage which is constructed from four bars connected in a loop by four rotating joints where one link of the chain is fixed.
The resulting mechanism has one degree of freedom.
Four-bar linkages are extensively studied in mechanics as well as numerical nonlinear algebra \cite{wampler2011numerical}.
Here, we extend a four-bar linkage to an elastic tensegrity framework by introducing two nodes which are attached to the two non-fixed joints by elastic cables.
Formally, we introduces six nodes with coordinates $p_1,\ldots,p_6 \in \R^2$, bars $B=\{12,23,34,41\}$ and elastic cables $C=\{35,46\}$. See Figure \ref{fig:fourbar-abstract} for an illustration of this basic setup.

 \begin{figure}[H]
     \centering
     \includegraphics[width=0.45\textwidth]{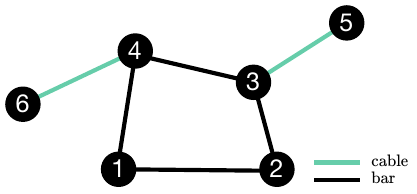}
     \caption{Setup of a four-bar elastic tensegrity framework.}
     \label{fig:fourbar-abstract}
 \end{figure}

The zero set of the bar constraints $b_{ij}$, $ij \in B$, is a curve of degree 6 which can be parameterized by the plane curve traced out by the motion of the midpoint $(p_3+p_4)/2$. In kinematics terminology the midpoint is a \emph{coupler point} and the plane curve is called the \textit{coupler curve} of the \textit{mechanism}.

The idea is to fix nodes 1, 2 and 5, and control node 6. For our model this means choosing $X=\{(p_{31}, p_{32}, p_{41}, p_{42})\} \simeq \mathbb{R}^4$ as internal parameters and
$ \Omega = \{ (p_{61}, p_{62}) \} \simeq \mathbb{R}^2$ as control parameters. Furthermore, we fix nodes $p_1 = (-1,0)$, $p_2=(1,0)$, $p_5=(4,3)$, bar lengths $l_{23}=3$, $l_{34}=1$, $l_{14}=1.5$, resting lengths $r_{35}=r_{46}=0.1$ and elasticities $c_{35}=1$, $c_{46}=2$.

In this setup the framework has an equilibrium degree of 64.
The resulting catastrophe discriminant $\cdisc$ is a curve of degree 288. $\cdisc$ and the catastrophe set $\cset$ are depicted in Figure \ref{fig:fourbar_catastrophes}.
The typical sizes of the stability set $\mathcal{S}_\gamma$, $\gamma \in \Omega$, are 2, 3 and 4.

\begin{figure}[H]
    \centering
    \includegraphics[width=0.9\textwidth]{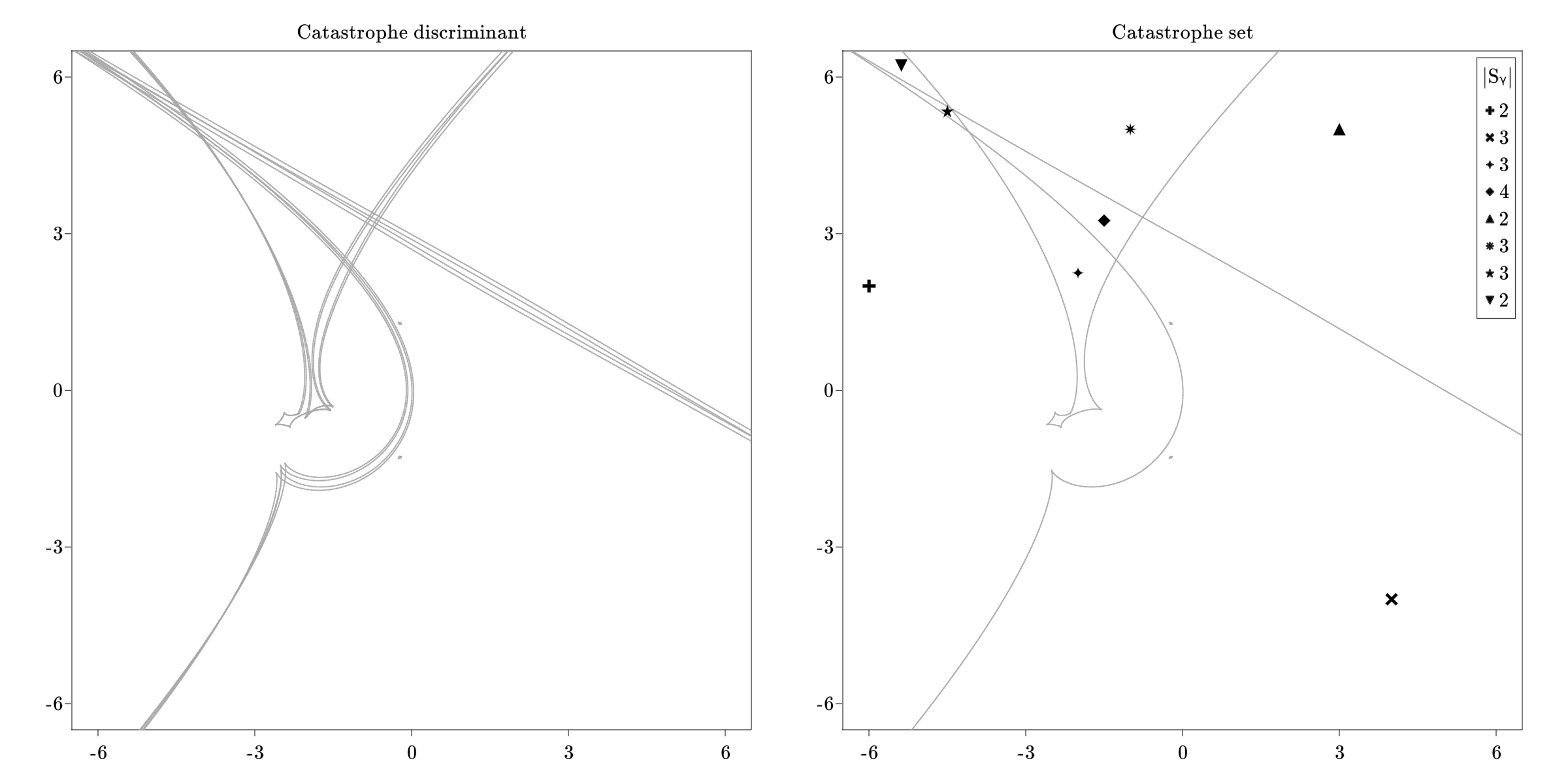}
    \caption{The catastrophe discriminant (left) and catastrophe set (right) of the elastic four bar framework. The cardinality of the stability set for points in each chamber of the complement of the catastrophe set is shown in the upper right corner.} 
    \label{fig:fourbar_catastrophes}
\end{figure}

\begin{figure}[H]
    \centering
    \includegraphics[width=0.90\textwidth]{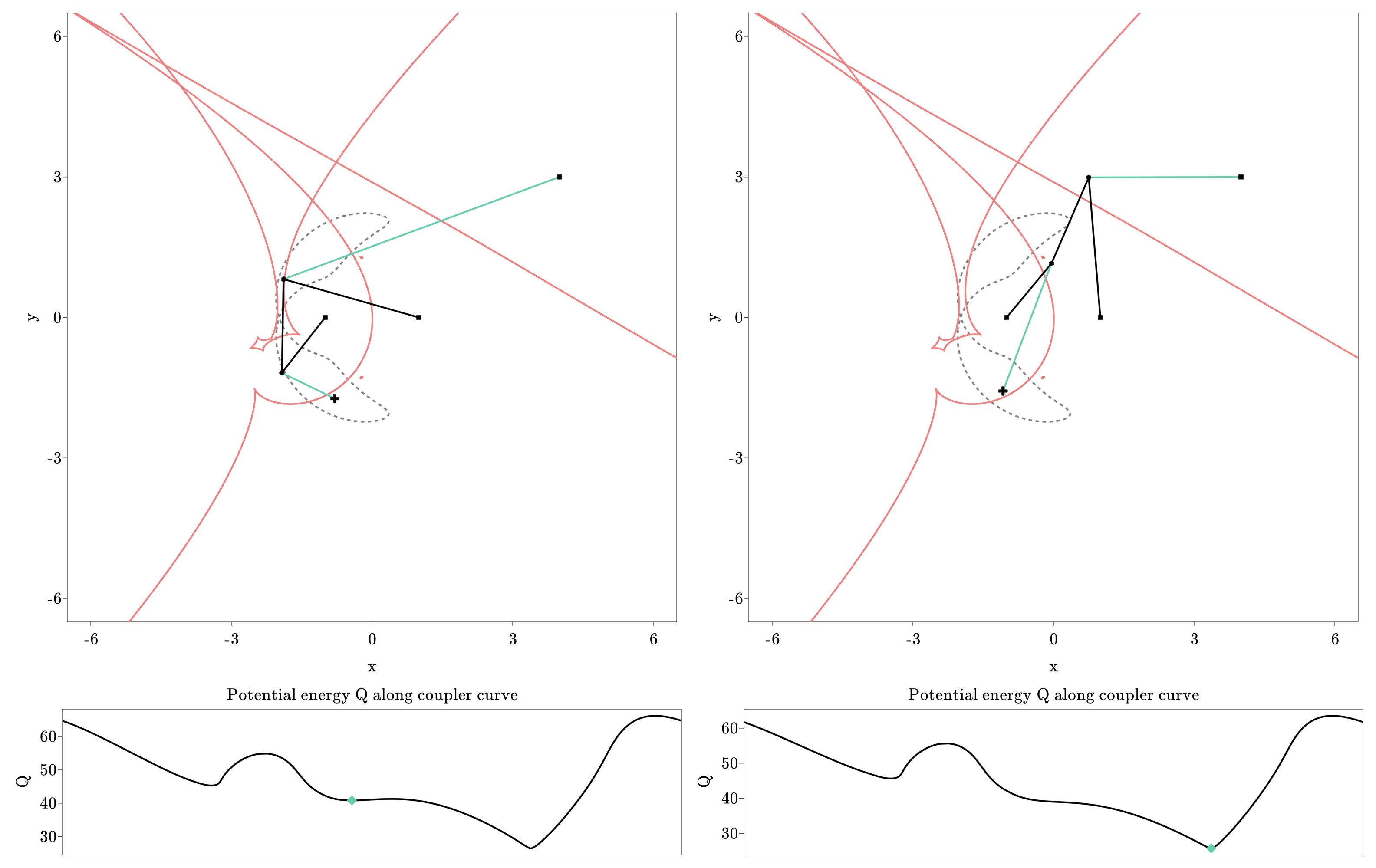}
    \caption{\footnotesize Left: The elastic four bar framework in a stable configuration. Right: Configuration of the framework after crossing the catastrophe set.
    The gray dashed line is the coupler curve of the four bar linkage traced out by the midpoint of the bar connecting node 2 and 3. The coupler curve parameterizes all possible four bar positions.
    The catastrophe set $\cset$ is depicted in red.
    At the bottom are the energy landscapes along the coupler curve with the current position depicted in green.}
    \label{fig:fourbar_before_after}
\end{figure}

Finally, we also want to give in Figure \ref{fig:fourbar_before_after} another concrete example of a catastrophe. There, the control node 5 is depicted by a cross and it is dragged in a straight line between its position in the left figure and its position in the right figure.
When the control node crosses the catastrophe set $\cset$, its previously stable position disappears from $\mathcal{S}_y$, and the framework ``jumps'' to a new position. Again, without knowledge of $\cset$ these catastrophes are extremely surprising. With knowledge of $\cset$ and Theorem \ref{theorem:path-lifting-correct} they become avoidable.

\section{Conclusion and future work}

This article described \textit{elastic tensegrity frameworks} as a large family of simple models based on Hooke's law and energy minimization. For this family we showed how to explicitly calculate and track all stable equilibrium positions of a given framework. More importantly, we showed how to calculate the catastrophe set $\cset$ by using pseudo-witness sets to encode a superset $\cdisc \supset \cset$. To do this we reformulated the problem algebraically to take advantage of tools in numerical nonlinear algebra. Knowing the catastrophe set provides extremely useful information, since Theorem \ref{theorem:path-lifting-correct} shows that paths of control parameters avoiding $\cset$ will also avoid discontinuous loss of equilibrium, and hence avoid surprising large-scale shape changes.

In our two illustrative examples, we chose the controls $\Omega$ as a two-dimensional space overlaid with the configuration itself. These choices were made to demonstrate the ideas. However, the calculation and tracking of all stable local minima by parameter homotopy and the encoding of $\cdisc \supset \cset$ by pseudo-witness sets apply much more generally. The control set $\Omega$ can be chosen in any way, and all the same methods apply, even if there are no easy visualizations for the controls desired. Therefore, for more complicated sets of control parameters $\Omega$ it is of interest to develop more efficient local sampling techniques based on Monte Carlo methods \cite{zappa2018monte}, perhaps only sampling $\cset$ locally near the initial configuration $(x(0),y(0))$ or locally near the intended path $y([0,1])$. For example, it may be enough to know only the points of $\cset$ nearest to a given initial or current position $y(t)$.

We would also mention recent work \cite{bernal2020machine} which details a sampling scheme whose goal is to learn the real discriminant of a parametrized polynomial system, as well as the number of real solutions on each connected component. They combine homotopy continuation methods with $k$-nearest neighbors and deep learning techniques. For elastic tensegrity frameworks, these techniques might be used to learn $\cdisc \cap \Omega_\mathbb{R}$.

Finally we discuss the potential of our results for use in \textit{mechanobiology} \cite{IngberWangStamenovic2014TensegrityCellBioPhysicsMechanicsReview},
where scientists have  frequently and successfully used tensegrity to model cell mechanics. Even small and simple elastic tensegrity frameworks (e.g. with 6 or 12 rigid bars, plus more cables) have been used to explain and predict experimental results observed in actual cells and living tissue
\cite{DeSantis-2011-cell-elasticity-tensegrity-model, volokh-2000-tensegrity-model-predicts-linear-stiffening-living-cells, Coughlin-stamenovic-2003-prestressed-cable-model-cytoskeleton, wang-stamenovic-2002-cell-prestress-adherent-contractile-cells}.
However, the tensegrity paradigm is not universally accepted in mechanobiology in part because it is viewed as a static theory, unable to explain dynamic, time-dependent phenomena \cite[see pages 13-16]{IngberWangStamenovic2014TensegrityCellBioPhysicsMechanicsReview}. It is here where catastrophe sets could play a role. We believe \textit{qualitative} phenomena observed in actual experiments could be predicted or explained by elastic tensegrity frameworks.
Knowing the catastrophe set for a simple tensegrity model with a biology-informed choice of $\Omega$ would give catastrophe predictions that could then be tested experimentally.

\bibliographystyle{plain}
\bibliography{references}

\bigskip \bigskip

\noindent
\footnotesize {\bf Authors' addresses:}

\smallskip
\noindent Alex Heaton, Technische Universit\"at Berlin and Max Planck Institute for Mathematics in the Sciences, {\tt heaton@mis.mpg.de, alexheaton2@gmail.com}

\noindent Sascha Timme, Technische Universit\"at Berlin, {\tt timme@math.tu-berlin.de}

\end{document}